\documentclass[12pt]{amsart}
\usepackage{amsfonts}
\usepackage{mathrsfs}
\usepackage{amscd,amsmath,amssymb,amsfonts}
\usepackage{color}
\usepackage[all]{xy}
\theoremstyle{plain}
\newtheorem{thm}{Theorem}
\newtheorem{lem}[thm]{Lemma}

\newtheorem{prop}[thm]{Proposition}

\theoremstyle{definition}
\newtheorem{defn}[thm]{Definition}

\newtheorem{question}[thm]{Question}

\newtheorem{nota}[thm]{Notation}

\numberwithin{thm}{section} \numberwithin{equation}{section}

\newcommand{\ga}[2]{\begin{gather}\label{#1}#2 \end{gather}}

\newcommand{\sC}{{\mathcal C}}
\newcommand{\sD}{{\mathcal D}}
\newcommand{\sE}{{\mathcal E}}
\newcommand{\sF}{{\mathcal F}}

\newcommand{\sH}{{\mathcal H}}

\newcommand{\sL}{{\mathcal L}}

\newcommand{\sO}{{\mathcal O}}
\newcommand{\sP}{{\mathcal P}}
\newcommand{\sQ}{{\mathcal Q}}
\newcommand{\sR}{{\mathcal R}}

\newcommand{\sU}{{\mathcal U}}

\newcommand{\sW}{{\mathcal W}}

\def\wt#1{\widetilde {#1}}

\begin{document}

\title{Globally F-regular type of Moduli spaces}
\author{Xiaotao Sun}
\address{Center of Applied Mathematics, School of Mathematics, Tianjin University, No.92 Weijin Road, Tianjin 300072, P. R. China}
\email{xiaotaosun@tju.edu.cn}
\author{Mingshuo Zhou}
\address{Center of Applied Mathematics, School of Mathematics, Tianjin University, No.92 Weijin Road, Tianjin 300072, P. R. China}
\email{zhoumingshuo@amss.ac.cn}
\date{September 18, 2019}
\thanks{Both authors are supported by the National Natural Science Foundation of China No.11831013;
Mingshuo Zhou is also supported by the National Natural Science Foundation of China No.11501154}
\begin{abstract} We prove moduli spaces of semistable parabolic bundles and generalized parabolic sheaves
with fixed determinant on a smooth projective curve are globally $F$-regular type.
\end{abstract}
\keywords{Frobenius split,  Moduli spaces, Parabolic sheaves}
\subjclass{Algebraic Geometry, 14H60, 14D20}
\maketitle

\section{Introduction}

Let $X$ be a variety over a perfect field $k$ of characteristic $p>0$ and $F:X\to X$ be the Frobenius morphism. The $X$ is called
$F$-split (Frobenius split) if the natural homomorphism $\sO_X\hookrightarrow F_*\sO_X$ is split. Although most of
projective varieties are not $F$-split, some important varieties are $F$-split. For example,  flag varieties and
their Schubert subvarieties (cf. \cite{MeRa}, \cite{RR}), the product of
two flag varieties for the same group $G$ (cf. \cite{MeRa88}) and cotangent bundles of flag varieties (cf. \cite{KLT}) are proved to be $F$-split.
An example, which is more closer to this article, should be mentioned. Mehta-Ramadas proved in \cite{MeR} that for a generic
smooth projective curve $C$ of genus $g$ over an algebraically closed field of characteristic $p\ge 5$, the moduli space of semistable parabolic bundles of rank $2$ on $C$ is $F$-split, and made conjecture that moduli spaces of semistable
parabolic bundles of rank $2$ on any smooth curve $C$ with a fixed determinant is F-split.

The notion of globally $F$-regular variety was introduced by K. E. Smith in \cite{Sm}, a variety $X$ is called globally $F$-regular if for any effective
divisor $D$, the natural homomorphism $\sO_X\hookrightarrow F^e_*\sO_X(D)$ is split for some integer $e>0$. It is clear that globally $F$-regular varieties must be
$F$-split. Also, some well-known F-split varieties include toric varieties and Schubert varieties are proved (\cite{Sm}, \cite{LRT}) to be globally F-regular. Thus it is natural to extend Mehta-Ramadas conjecture: the moduli spaces $\sU^L_{C,\,\omega}$ of semistable parabolic bundles of rank $r$ with a
fixed determinant $L$ on any smooth curves $C$ (parabolic structures determined by a given data $ω$) are globally F-regular varieties. It remains
to be a very difficult open problem, we will study its characteristic zero analogy in this article.

A variety $X$ over a field of characteristic zero is called globally $F$-regular type (resp. $F$-split type) if its modulo $p$ reduction $X_p$ is
globally $F$-regular (resp. $F$-split) for a dense set of $p$. Projective varieties $X$, which are globally F-regular type, have remarkable geometric and
cohomological properties: (1) $X$ must be normal, Cohen-Macaulay with rational singularities, and must have log terminal singularities if it is $\mathbb{Q}$-Gorenstein;
(2) $H^i(X,\sL)=0$ for $i>0$ and nef line bundle $\sL$.

Let $\sU_{C,\,\omega}$ be moduli spaces of semistable parabolic bundles of rank $r$ and degree $d$ on smooth curves $C$ of genus $g\ge 0$ with parabolic
structures determined by $\omega=(k,\{\vec n(x), \vec a(x)\}_{x\in I})$ and
$${\rm det}: \sU_{C,\,\omega}\to J^d_C$$
be the determinant morphism. For any $L\in J^d_C$, the fiber $$\sU_{C,\,\omega}^L:={\rm det}^{-1}(L)$$ is called moduli space of semistable parabolic bundles with
a fixed determinant $L$. Then the first main result in this article is

\begin{thm}[See Theorem \ref{thm3.7}]\label{thm1.1} The moduli spaces $\sU_{C,\,\omega}^L$ are of globally $F$-regular type.
\end{thm}

When the projective curve $C$ has exactly one node (irreducible, or reducible), the moduli space $\sU_{C,\,\omega}$ is not normal and its normalization
is a moduli space $\sP_{\omega}$ of semistable generalized parabolic sheaves (GPS) on $\wt C$ (where $\wt{C}$ is normalization of $C$). There exist a similar determinant morphism
${\rm det}:\sP_{\omega}\to J^d_{\wt{C}}$. For any $L\in J^d_{\wt{C}}$, the fiber
$$\sP^L_{\omega}:={\rm det}^{-1}(L)$$
is called a moduli space of semistable generalized parabolic sheaves (GPS) with a fixed determinant $L$ on $\wt C$. Then the second main result
in this article is
\begin{thm}[See Theorem \ref{thm3.14} and Theorem \ref{thm3.22}]\label{thm1.2} The moduli spaces $\sP_{\omega}^L$ are of globally $F$-regular type.
\end{thm}

To describe the idea of proof, recall that the moduli space $\sU^L_{C,\,\omega}$ is
a GIT quotient $(\sR^{ss}_{\omega})^L//{\rm SL}(V)$, where $(\sR^{ss}_{\omega})^L\subset\sR^L_F$ is a open set of a quasi-projective variety $\sR^L_F$ ( i.e. the
set of GIT semistable points respect to a polarization $\Theta_{\sR,\,\omega}$ determined by $\omega$). Then our idea is to find
a flag bundle ${\sR'}^L_F\xrightarrow{\hat f}\sR^L_F$ over $\sR^L_F$ and a data $\omega'$ such that $\sU^L_{C,\,\omega'}=({\sR'}_{\omega'}^{ss})^L//{\rm SL}(V)$
is a Fano variety with an open subvariety $X\subset\sU^L_{C,\,\omega'}$ and a morphism $X\xrightarrow{f}\sU^L_{C,\,\omega}$ satisfying
$f_*\sO_X=\sO_{\sU^L_{C,\,\omega}}.$
Since Fano varieties are globally $F$-regular type by Proposition 6.3 of \cite{Sm}, so are $X$ and $\sU^L_{C,\,\omega}$ if the equality $f_*\sO_X=\sO_{\sU^L_{C,\,\omega}}$
commutes with modulo $p$ reductions for a dense set of $p$. To prove that $f_*\sO_X=\sO_{\sU^L_{C,\,\omega}}$
commutes with modulo $p$ reductions for a dense set of $p$, one has to show in particular that a GIT quotient over $\mathbb{Z}$ must commute
with modulo $p$ reductions for a dense set of $p$, which is Lemma \ref{lem2.9} (we thought at first that Lemma \ref{lem2.9} must be well-known to experts,
but we are not able to find any reference). We formulate our idea in Proposition \ref{prop2.10} in a general setting, the proof of Theorem \ref{thm1.1} and Theorem \ref{thm1.2} becomes to check
conditions in Proposition \ref{prop2.10}.

We describe briefly content of the article. In Section 2, we collect notions and properties of globally $F$-regular type varieties, in particular,
we formulate and prove Proposition \ref{prop2.10}, which is our technical tool to show globally $F$-regular type of GIT quotients. In Section 3, we recall some facts about moduli spaces of parabolic bundles and prove Theorem \ref{thm1.1}. Finally, we prove Theorem \ref{thm1.2} in Section 4.

{\it Acknowledegements:} Xiaotao Sun would like to thank C. S. Seshadri for a number of emails of discussions about Lemma \ref{lem2.9}, and
he also would like to thank K. Schwede and K. E. Smith for discussions (by emails) of globally $F$-regular type varieties.

\section{Globally F-regular varieties}

We collect firstly some notions and facts of globally F-regular varieties over a perfect field $k$ of positive characteristic and recall the definition of
globally F-regular type of varieties over a field of characteristic zero. Our main references here are \cite{Br}, \cite{Sc} and \cite{Sm}.

Let $X$ be a variety over a perfect field $k$ of $char (k)=p>0$, $$F:X\to X$$
be the Frobenius map and $F^e:X\to X$
be the e-th iterate of Frobenius map. When $X$ is normal, for any (weil) divisor $D\in Div(X)$,
$$\sO_X(D)(V)=\{\,f\in K(X)\,|\, div_V(f)+D|_V\ge 0\,\},\quad \forall\,\,V\subset X$$
is a reflexive subsheaf of constant sheaf $K=K(X)$. In fact, we have
$$\sO_X(D)=j_*\sO_{X^{sm.}}(D)$$
where $j:X^{sm.}\hookrightarrow X$ is the open set of smooth points, and $\sO_X(D)$ is an invertible sheaf
if and only if $D$ is a Cartier divisor.

\begin{defn}\label{defn2.1} A normal variety $X$ over a perfect field is called
\emph{stably Frobenius $D$-split} if $\sO_X\to F^e_*\sO_X(D)$ is split for some $e>0$.
$X$ is called \emph{globally F-regular} if $X$ is stably Frobenius $D$-split
for any effective divisor $D$.
\end{defn}

The advantage of this definition is that any open set $U\subset X$ of a globally F-regular variety $X$ is globally F-regular. Its disadvantage
is the requirement of normality of $X$. When $X$ is not normal, one possible remedy of Definition \ref{defn2.1} is to require that $D$ is a Cartier divisor. Then it loses the
advantage that any open set $U\subset X$ is globally F-regular since a Cartier divisor on $U$
may not be extended to a Cartier divisor on $X$. But, when $X$ is a projective variety and is stably Frobenius $D$-split for any effective Cartier $D$, $X$ must be
normal and Cohen-Macaulay according to K. E. Smith (Theorem 3.10 and Theorem 4.1 of \cite{Sm}).

\begin{prop}[Theorem 3.10 of \cite{Sm}]\label{prop2.2} Let $X$ be a projective variety over a perfect field. Then the following statements are equivalent.
\begin{itemize} \item [(1)] $X$ is normal and is stably Frobenius $D$-split for any effective $D$;
\item [(2)] $X$ is stably Frobenius $D$-split for any effective Cartier $D$;
\item [(3)] For any ample line bundle $\sL$, the section ring of $X$
$$R(X,\sL)=\bigoplus_{n=0}^{\infty}H^0(X,\sL^n)$$
is strongly F-regular.
\end{itemize}
\end{prop}

\begin{proof} It is clear that $(1)\Rightarrow (2)$, and $(2)\Rightarrow (3)$ is proved in Theorem 3.10 of \cite{Sm}. That $(3)\Rightarrow (1)$ is a modification of the proof
in \cite{Sm}. By Theorem 4.1 of \cite{Sm}, $X$ is normal and Cohen-Macaulay.  Let $X^{sm.}\subset X$ be the open set of smooth points, then $R(X,\sL)=R(X^{sm.},\sL)$ and, for any effective $D\in Div(X)$,
$D\cap X^{sm.}$ is an effective Cartier  divisor on $X^{sm.}$. Then the proof of $(1)\Rightarrow (3)$ in Theorem 3.10 of \cite{Sm} implies
that $X^{sm.}$ is stably Frobenius $D\cap X^{sm.}$-split, which implies that $X$ is stably Frobenius $D$-split.
\end{proof}

A variety $X$ is called \emph{F-split} if $\sO_X\to F_*\sO_X$ is split. In particular,
\emph{globally F-regular} varieties are \emph{F-split}. Let $X\xrightarrow{f}Y$ be a morphism such that $f_*\sO_X=\sO_Y$, then any splitting map
$F_*\sO_X\xrightarrow{\psi}\sO_X$ of $\sO_X\to F_*\sO_X$ induces a splitting map $F_*\sO_Y=F_*f_*\sO_X=f_*F_*\sO_X\xrightarrow{f_*\psi} f_*\sO_X=\sO_Y$. There is a generalization
of above useful observation.

\begin{lem}[Corollary 6.4 of \cite{Sc}]\label{lem2.3} Let $f:X\to Y$ be a morphism of varieties over a perfect field $k$ of $char(k)=p>0$.
If the natural map $\sO_Y\xrightarrow{i} f_*\sO_X$ splits and $X$ is globally F-regular, then $Y$ is stably Frobenius $D$-split for any effective Cartier divisor $D$, and
it is globally F-regular when $Y$ is normal.
\end{lem}

\begin{proof} For any Cartier divisor $D\in Div(Y)$ defined by a section $s\in\Gamma(Y,\sO_Y(D))$, let $H=f^*D$ and $F_*^e\sO_X(H)\xrightarrow{h}\sO_X$ be a splitting of
$\sO_X\to F^e_*\sO_X\xrightarrow{F^e_*f^*(s)} F_*^e\sO_X(H)$, and
 $f_*\sO_X\xrightarrow{j}\sO_Y$ be a splitting of $\sO_Y\xrightarrow{i} f_*\sO_X$.
Then $\sO_Y(D)\xrightarrow{1\otimes i}\sO_Y(D)\otimes f_*\sO_X=f_*\sO_X(H)$ induces
$$F^e_*\sO_Y(D)\xrightarrow{F^e_*1\otimes i}F^e_*f_*\sO_X(H)=f_*F^e_*\sO_X(H)\xrightarrow{f_*h}f_*\sO_X\xrightarrow{j}\sO_Y$$
is a splitting of $\sO_Y\to F^e_*\sO_Y\xrightarrow{F^e_*s} F^e_*\sO_Y(D)$. When $Y$ is normal, let $Y_0\subset Y$ be the open set of smooth points, $Y$ is globally F-regular if and only if $Y_0$ is stably
Frobenius $D$-split for any effective Cartier divisor $D\in Div(Y_0)$, which is true by applying above argument to $f^{-1}(Y_0)\xrightarrow{f} Y_0$.
\end{proof}

For any scheme $X$ of finite type over a field $K$ of
characteristic zero, there is a
finitely generated $\mathbb{Z}$-algebra $A\subset K$ and an $A$-flat
scheme $$X_A\to S={\rm Spec}(A)$$ such that $X_K=X_A\times_S{\rm
Spec}(K)\cong X$. $X_A\to S={\rm Spec}(A)$ is called an integral model of $X/K$, and a closed fiber
$X_s=X_A\times_S{\rm Spec}(\overline{k(s)})$ is
called "\textbf{modulo $p$ reduction of $X$}" where $p={\rm
char}(k(s))>0$.

\begin{defn}\label{defn2.4} A variety $X$ over a field of characteristic zero is said to be of \emph{globally F-regular type} (resp.\emph{ F-split type}) if its "\textbf{modulo $p$ reduction of $X$}"  are globally F-regular (resp. \emph{F-split })
for a dense set of $p$.
\end{defn}

Projective varieties of \emph{globally F-regular type} have many nice properties and a good vanishing theorem of cohomology.

\begin{thm}[Corollary 5.3 and Corollary 5.5 of \cite{Sm}]\label{thm2.5} Let $X$ be a projective variety over a field of characteristic zero. If $X$ is of globally F-regular type,
then we have \begin{itemize} \item [(1)] $X$ is normal, Cohen-Macaulay with rational singularities. If $X$ is $\mathbb{Q}$-Gorenstein, then $X$ has log terminal singularities.
\item [(2)] For any nef line bundle $\sL$ on $X$, we have $H^i(X,\sL)=0$ when $i>0$. In particular, $H^i(X,\sO_X)=0$ whenever $i>0$.
\end{itemize}
\end{thm}

A normal projective variety $X$ is called a \emph{Fano variety} if
$$\omega_X^{-1}=\sH om_{\sO_X}(\omega_X,\sO_X)$$ is an ample line bundle. One of important examples of \emph{globally F-regular type} varieties is

\begin{prop}\label{prop2.6} (\cite[Proposition~6.3]{Sm})  A Fano variety (over a field of characteristic zero) with at most rational singularities is of globally F-regular type.
\end{prop}

To show that a variety $Y$ is of \emph{globally F-regular type}, one possible approach is to construct an open set $X$ of a Fano variety (thus $X$ is of \emph{globally F-regular type}) with a morphism $f:X\to Y$ such that $f_*\sO_X=\sO_Y$. Then $Y$ is of \emph{globally F-regular type} if the following characteristic zero analogy of Lemma \ref{lem2.3} is true.

\begin{question}\label{question2.7} Let $X\xrightarrow{f} Y$ be a morphism of varieties over a field $K$ of ${\rm char}(K)=0$ such that $\sO_Y\to f_*\sO_X$ is split and $X$ is of
\emph{globally F-regular type}. Is $Y$ a variety of \emph{globally F-regular type} ?
\end{question}

Let $f_*\sO_X\xrightarrow{\beta}\sO_Y$ be a splitting of $\sO_Y\to f_*\sO_X$. Then Question \ref{question2.7} consists: (1)
Can we choose a model $f_A:X_A\to Y_A$ of $f:X\to Y$ such that the $\sO_Y$-homomorphism  $(f_{A_*}\sO_{X_A})\otimes_AK\xrightarrow{\beta}\sO_{Y_A}\otimes_AK$ can be extended to
$f_{A_*}\sO_{X_A}\xrightarrow{\beta_A}\sO_{Y_A}$ ? (2) Is there a dense set of closed point ${\rm Spec}(\overline{k(s)})\to S={\rm Spec}(A)$ such that
$i^*_sf_{A*}\sO_{X_A}=f_{s*}j^*_s\sO_{X_A}$ ?
where $Y_s=Y_A\times_A\overline{k(s)}\xrightarrow{i_s} Y_A$, $X_s=X_A\times_A\overline{k(s)}\xrightarrow{j_s} X_A$ and
$$\xymatrix{
  X_s \ar[d]_{f_s} \ar[r]^{j_s}
                &  X_A\ar[d]^{f_A}  \\
  Y_s \ar[r]^{i_s}
                & Y_A    .}$$

\begin{defn}\label{defn2.8} A morphism $X\xrightarrow{f} Y$ of varieties over a field $K$ of ${\rm char}(K)=0$ is called \emph{$p$-compatible} if there is an integral model
$X_A\xrightarrow{f_A} Y_A$ such that $i^*_sf_{A*}\sO_{X_A}=f_{s*}j^*_s\sO_{X_A}$ for $s\in {\rm Spec}(A)$.
\end{defn}

It is clear that (1) has an affirmative answer when either $f_*\sO_X$ is a coherent $\sO_Y$-module or the splitting map $\beta: f_*\sO_X\to\sO_Y$
is a homomorphism of $\sO_Y$-algebras. (2) has an affirmative answer for flat morphism $f:X\to Y$ with coherent $R\,^if_*\sO_X$ ($i\ge 0$). It is also clear that any affine morphism must be $p$-compatible. When $X$, $Y$ are open set of GIT quotients and $f: X\to Y$ is induced by a $G$-invariant $p$-compatible morphism $\hat f:\sR'\to \sR$ of parameter spaces, we will show that $f: X\to Y$ is $p$-compatible morphism in Proposition \ref{prop2.10}, which will need the following lemma.

\begin{lem}\label{lem2.9}
Let $X\to S={\rm Spec}(A)$ be a flat projective morphism, $A$ be an integral $\mathbb{Z}$-algebra of finite type and $G\to S$ be a $S$-flat reductive group scheme
with action on $X$ over $S$. If $L$ is a relative ample line bundle on $X$ linearizing the action of $G$, let
$$X^{ss}(L)\xrightarrow{\pi} X^{ss}(L)//G:=Y$$
be the GIT quotient over $S$. Assume that the geometrically generic fiber of $X^{ss}(L)\to S$ is an irreducible normal variety.
Then there is a dense open set $U\subset S$ such that for any $s\in U$
$$Y\times_S\overline{k(s)}\cong X_s^{ss}(L_s)//G_s$$
where $X_s=X\times_S\overline{k(s)}$ (resp. $G_s=G\times_S\overline{k(s)}$) is the geomerically closed fiber of $X\to S$ (resp. $G\to S$) at ${\rm Spec}(\overline{k(s))}\to S$.
\end{lem}

\begin{proof} Let $X^{ss}(L)\times_S\overline{k(s)}\xrightarrow{\pi_s}Y_s$ be the pullback of $X^{ss}(L)\xrightarrow{\pi}Y$ under the base change
${\rm Spec}(\overline{k(s))}\to S$. By Proposition 7 of \cite{Se},
$$X^{ss}(L)\times_S\overline{k(s)}=X_s^{ss}(L_s).$$
Then there is a unique $\overline{k(s)}$-morphism $X^{ss}_s(L_s)//G_s\xrightarrow{\theta} Y_s$ such that
$$\xymatrix{
   X^{ss}_s(L_s) \ar[dr]^{\pi_s} \ar[r]^{}
                & X_s^{ss}(L_s)//G_s \ar[d]^{\theta}  \\
                & Y_s  }$$
is commutative. Let $\overline{Y_s}:=X^{ss}_s(L_s)//G_s$, $k=\overline{k(s)}$, it is known that $\theta$ induces a bijective map
$\overline{Y_s}(k)\xrightarrow{\theta} Y_s(k)$ on the sets of $k$-points (cf. Proposition 9 (i) of \cite{Se}). By the assumption, geometrically generic fiber of $Y\to S$ is an irreducible normal projective variety. Thus there is a dense open set $U\subset S$ such that any closed point ${\rm Spec}(\overline{k(s))}\to U$ satisfies
(1) $Y_s$ is normal, and (2) the morphism $X_s^{ss}(L_s)\xrightarrow{\pi_s} Y_s$ is generic smooth, where (1) is (iv) of Th{\'e}or{\'e}me (12.2.4) in \cite{Gr} and (2) holds since
$K=Q(A)=k(S)$ is a field of characteristic zero. Then generic smoothness of $\pi_s$ implies the generic smoothness of $\overline{Y_s}\xrightarrow{\theta} Y_s$, which must be
an isomorphism by Zariski main theorem since $Y_s$ is normal.
\end{proof}

Let $(\hat Y,L)$, $(\hat Z,L')$ be polarized projective varieties over an algebraically closed field $K$ of characteristic zero with actions of a reductive group scheme $G$ over $K$, and $\hat Y^{ss}(L)\subset \hat Y$
(resp. $\hat Y^s(L)\subset \hat Y^{ss}(L)$) be the open set of GIT semi-stable (resp. GIT stable) points of $\hat Y$. Then there are projective GIT quotients
\ga{2.1} {\hat Y^{ss}(L)\xrightarrow{\psi} Y:=\hat Y^{ss}(L)//G,\quad \hat Z^{ss}(L')\xrightarrow{\varphi} Z:=\hat Z^{ss}(L')//G.}

\begin{prop}\label{prop2.10} Let $Z$, $Y$ be the GIT quotients in \eqref{2.1}. Assume
\begin{itemize}
\item [(1)] there are $G$-invariant normal open subschemes $\sR\subset \hat Y$, $\sR'\subset \hat Z$ such that $\hat Y^{ss}(L)\subset\sR$, $Z^{ss}(L')\subset \sR'$;
\item [(2)] there is a $G$-invariant $p$-compatible morphism $\sR'\xrightarrow{\hat f}\sR$ such that $\hat f_*\sO_{\sR'}=\sO_{\sR}$;
\item [(3)] there is an $G$-invariant open set $W\subset Z^{ss}(L')$ such that
$${\rm Codim}(\sR'\setminus W)\ge 2, \quad \hat X=\varphi^{-1}\varphi(\hat X)$$ where $\hat X=W\cap \hat f^{-1}(\hat Y^{ss}(L))$.\end{itemize}
If $Z$ is of globally F-regular type. Then so is $Y$.
\end{prop}

\begin{proof} Let $X=\varphi(\hat X)\subset Z$, which is an open set of $Z$ since
$$\varphi(Z^{ss}(L')\setminus\hat X)=Z\setminus X$$
by the condition $\varphi^{-1}(X)=\hat X$ and that $Z^{ss}(L')\setminus\hat X$ is a $G$-invariant closed subset. There is a morphism
$X\xrightarrow{f} Y$ such that
$$\xymatrix{
  \hat X \ar[d]_{\hat f|_{\hat X}} \ar[r]^{\varphi}
                &  X\ar[d]^{f}  \\
  \hat Y^{ss}(L) \ar[r]^{\psi}
                & Y            }$$ is commutative. For any open set $U\subset Y$, since $\hat f_*\sO_{\sR'}=\sO_{\sR}$, we have
$$\aligned\sO_Y(U)&=\sO_{\sR}(\psi^{-1}(U))^{inv.}=\sO_{\sR'}(\hat f^{-1}\psi^{-1}(U))^{inv.}\\
&=\sO_{\sR'}(W\cap \hat f^{-1}\psi^{-1}(U))^{inv.}=\sO_{\hat X}(\hat f|_{\hat X}^{-1}\psi^{-1}(U))^{inv.}\\&=
\sO_{\hat X}(\varphi^{-1}f^{-1}(U))^{inv.}=\sO_X(f^{-1}(U))=f_*\sO_X(U)\endaligned$$
where the third equality holds because $\hat f^{-1}\psi^{-1}(U)\setminus W\cap \hat f^{-1}\psi^{-1}(U)=\hat f^{-1}\psi^{-1}(U)\cap (\sR'\setminus W)$
has codimension at least two. Thus we have
\ga{2.2} {\sO_Y=f_*\sO_X,\,\,\text{where $X$ is of globally F-regular type.}}
To show that $Y$ is of globally F-regular type,  it is enough to show that the morphism $X\xrightarrow{f} Y$ is
$p$-compatible.

Let $(\hat Y_A,\sL)$, $(\hat Z_A,\sL')$ be integral models of $(\hat Y,L)$, $(\hat Z,L')$ with actions of a reductive group scheme $G_A$ over $S={\rm Spec}(A)$, and $\hat Y_A^{ss}(\sL)\subset \hat Y_A$
(resp. $\hat Y_A^s(\sL)\subset \hat Y_A^{ss}(\sL)$) be the open subscheme of GIT semi-stable (resp. GIT stable) points of $\hat Y_A$. Then there are GIT quotients
$$\hat Y_A^{ss}(\sL)\xrightarrow{\psi_A} Y_A:=\hat Y_A^{ss}(\sL)//G_A,\quad \hat Z_A^{ss}(\sL')\xrightarrow{\varphi_A} Z_A:=\hat Z_A^{ss}(\sL')//G_A,$$
which are projective over $S={\rm Spec}(A)$ and $\psi_A$, $\varphi_A$ are surjective $G_A$-invariant affine morphisms (cf. Theorem 4 of \cite{Se}).

We can choose $G_A$-invariant open subschemes
$\sR_A\subset \hat Y_A$, $\sR'_A\subset \hat Z_A$, $W_A\subset Z_A^{ss}(\sL')$, $X_A\subset Z_A$ and a $G_A$-invariant morphism $\sR_A'\xrightarrow{\hat f_A}\sR_A$ such that $\hat Y_A^{ss}(\sL)\subset\sR_A$, $Z_A^{ss}(\sL')\subset \sR_A'$,
$\hat f_{A*}\sO_{\sR_A'}=\sO_{\sR_A}$. Let
$$\hat X_A= \varphi_A^{-1}(X_A), \,\,\sR_s'=\sR'_A\times_A \overline{k(s)},\,\, \sR_s=\sR_A\times_A\overline{k(s)},$$ and $\hat f_s=\hat f_A\otimes \overline{k(s)}$ ($\forall\, s\in S$). Then we have $\hat f_{s*}\sO_{\sR_s'}=\sO_{\sR_s}$,
\ga{2.3} {{\rm Codim}(\sR_s'\setminus W_s)\ge 2,\,\, \hat X_s=W_s\cap \hat f_s^{-1}(\hat Y_A^{ss}(\sL)\times_A\overline{k(s)})} (by shrinking $S$) where $W_s=W_A\times_A\overline{k(s)}$, $\hat X_s=\hat X_A\times_A\overline{k(s)}$ and $$\hat Y_A^{ss}(\sL)\times_A\overline{k(s)}=\hat Y_s^{ss}(\sL_s),\quad \hat Z_A^{ss}(\sL')\times_A\overline{k(s)}=\hat Z_s^{ss}(\sL'_s)$$ (cf. Proposition 7 of \cite{Se}). Then, by Lemma \ref{lem2.9}, we have
$$Z_s=Z^{ss}_s(\sL'_s)//G_s, \quad Y_s=Y_s^{ss}(\sL_s)//G_s.$$ Thus, for any open sets $U\subset Z_s$, $V\subset Y_s$,  one has
$$\sO_{Z_s}(U)=\sO_{\sR'_s}(\varphi_s^{-1}(U))^{inv.},\quad \sO_{Y_s}(V)=\sO_{\sR_s}(\psi_s^{-1}(V))^{inv.}.$$
Recall $X_s\subset Z_s$, $\varphi_s^{-1}(X_s)=\hat X_s=W_s\cap \hat f_s^{-1}(\hat Y_s^{ss}(\sL_s))$ and consider
$$\xymatrix{
  \hat X_s \ar[d]_{\hat f_s} \quad \ar[r]^{\varphi_s}
                &  X_s\ar[d]^{f_s}  \\
  \hat Y_s^{ss}(\sL_s) \ar[r]^{\psi_s}
                & Y_s            }$$
we have $\sO_{Y_s}(V)=\sO_{\sR_s}(\psi_s^{-1}(V))^{inv.}=\sO_{\sR'_s}(\hat f_s^{-1}(\psi_s^{-1}(V)))^{inv.}$ since $\hat f_{s*}\sO_{\sR'_s}=\sO_{\sR_s}$.
Because the codimension of $$\hat f_s^{-1}(\psi_s^{-1}(V))\setminus W_s\cap \hat f_s^{-1}(\psi_s^{-1}(V))=\hat f_s^{-1}(\psi_s^{-1}(V))\cap (\sR'_s\setminus W_s)$$
is at least two, we have
$$\aligned \sO_{Y_s}(V)&=\sO_{\hat X_s}(\hat f_s^{-1}(\psi_s^{-1}(V)))^{inv.}=\sO_{\hat X_s}(\varphi_s^{-1}f_s^{-1}(V))^{inv.}\\&=
\sO_{X_s}(f^{-1}_s(V))=(f_s)_*\sO_{X_s}(V).\endaligned$$
Thus $\sO_{Y_s}=(f_s)_*\sO_{X_s}$, which implies that $f:X\to Y$ is $p$-compatible and $Y$ is of globally F-regular type since $X$ is so.

\end{proof}

\section{Globally $F$-regular type of Moduli spaces of parabolic bundles}

In this section, we prove that moduli spaces of parabolic bundles with a fixed
determinant on a smooth curve are of globally F-regular type.

Let $C$ be an irreducible projective curve of genus $g\ge 0$ over an
algebraically closed field $K$ of characteristic zero, which has at most
one node $x_0\in C$. Let $I$ be a finite set of smooth points of $C$, and
$E$ be a coherent sheaf of rank $r$ and degree $d$ on $C$ (the rank
$r(E)$ is defined to be dimension of $E_{\xi}$ at generic point
$\xi\in C$, and $d=\chi(E)-r(1-g)$).

\begin{defn}\label{defn3.1} By a quasi-parabolic structure of $E$ at a
smooth point $x\in C$, we mean a choice of flag of quotients
$$E_x=Q_{l_x+1}(E)_x\twoheadrightarrow
Q_{l_x}(E)_x\twoheadrightarrow\cdots\cdots\twoheadrightarrow
Q_1(E)_x\twoheadrightarrow Q_0(E)_x=0$$ of the fibre $E_x$, $n_i(x)={\rm
dim}(ker\{Q_i(E)_x\twoheadrightarrow Q_{i-1}(E)_x\})$ ($1\le i\le l_x+1$)
are called type of the flags. If, in addition, a sequence of integers
$$0\leq a_1(x)<a_2(x)<\cdots
<a_{l_x+1}(x)< k$$ are given, we call that $E$ has a parabolic
structure of type $$\vec n(x)=(n_1(x),n_2(x),\cdots,n_{l_x+1}(x))$$ and
weight $\vec a(x)=(a_1(x),a_2(x),\cdots,a_{l_x+1}(x))$ at $x\in C$.
\end{defn}

\begin{defn}\label{defn3.2} For any subsheaf $F\subset E$, let $Q_i(E)_x^F\subset
Q_i(E)_x$ be the image of $F$ and $n_i^F={\rm
dim}(ker\{Q_i(E)_x^F\twoheadrightarrow Q_{i-1}(E)_x^F\})$. Let
$${\rm par}\chi(E):=\chi(E)+\frac{1}{k}\sum_{x\in
I}\sum^{l_x+1}_{i=1}a_i(x)n_i(x),$$
$${\rm par}\chi(F):=\chi(F)+\frac{1}{k}\sum_{x\in
I}\sum^{l_x+1}_{i=1}a_i(x)n^F_i(x).$$
Then $E$ is called semistable (resp., stable) for $\omega=(k, \{\vec n(x),\,\,\vec a(x)\}_{x\in I})$ if for any
nontrivial $E'\subset E$ such that $E/E'$ is torsion free,
one has
$${\rm par}\chi(E')\leq
\frac{{\rm par}\chi(E)}{r}\cdot r(E')\,\,(\text{resp., }<).$$
\end{defn}

\begin{thm}[Theorem X1 of \cite{NR} or Theorem 2.13 of \cite{Su3} for arbitrary rank]\label{thm3.3} There
exists a seminormal projective variety
$$\sU_{C,\,\omega}:=\sU_C(r,d,
\{k,\vec n(x),\vec a(x)\}_{x\in I}),$$ which is the coarse moduli
space of $s$-equivalence classes of semistable parabolic sheaves $E$
of rank $r$ and $\chi(E)=\chi=d+r(1-g)$ with parabolic structures of type
$\{\vec n(x)\}_{x\in I}$ and weights $\{\vec a(x)\}_{x\in I}$ at
points $\{x\}_{x\in I}$. If $C$ is smooth, then it is normal, with
only rational singularities.
\end{thm}

Recall the construction of $\sU_{C,\,\omega}=\sU_C(r,d,\omega)$. Fix a line bundle $\sO(1)=\sO_C(c\cdot y)$ on $C$ of ${\rm deg}(\sO(1))=c$, let
$\chi=d+r(1-g)$, $P$ denote the polynomial $P(m)=crm+\chi$,
$\sO_C(-N)=\sO(1)^{-N}$ and $V=\Bbb C^{P(N)}$. Let $\bold Q$ be the Quot scheme of quotients $V\otimes\sO_{C}(-N)\to F\to 0$ (of rank
$r$ and degree $d$) on $C$. Thus there is on $C\times\bold Q$ a universal quotient
$$V\otimes\sO_{C\times\bold Q}(-N)\to \sF\to 0.$$
Let $\sF_x=\sF|_{\{x\}\times\bold Q}$ and $Flag_{\vec n(x)}(\sF_x)\to\bold Q$ be the relative flag scheme of type $\vec n(x)$. Let
$$\sR=\underset{x\in I}{\times_{\bold Q}}Flag_{\vec n(x)}(\sF_x)\to \bold Q,$$
on which reductive group ${\rm SL}(V)$ acts. The data $\omega=(k, \{\vec n(x),\,\,\vec a(x)\}_{x\in I})$, more precisely, the
weight $(k,\{\vec a(x)\}_{x\in I})$ determines a polarisation
$$\Theta_{\sR,\omega}=({\rm
det}R\pi_{\sR}\sE)^{-k}\otimes\bigotimes_{x\in I}
\lbrace\bigotimes^{l_x}_{i=1} {\rm det}(\sQ_{\{x\}\times
\sR,i})^{d_i(x)}\rbrace\otimes\bigotimes_q{\rm
det}(\sE_y)^{\ell}$$
on $\sR$ such that the open set $\sR^{ss}_{\omega}$ (resp. $\sR^s_{\omega}$) of
GIT semistable (resp. GIT stable) points are precisely the set of semistable (resp. stable) parabolic sheaves on $C$ (see \cite{Su3}), where $\sE$ is the pullback of $\sF$ (under
 $C\times\sR\to C\times \bold Q$), ${\rm
det}R\pi_{\sR}\sE$ is determinant line bundle of cohomology,
$$\sE_x=\sQ_{\{x\}\times \sR,l_x+1}\twoheadrightarrow\sQ_{\{x\}\times \sR,l_x}\twoheadrightarrow \sQ_{\{x\}\times \sR,l_x-1}
\twoheadrightarrow\cdots\twoheadrightarrow \sQ_{\{x\}\times
\sR,1}\twoheadrightarrow0$$ are universal quotients on $\sR$ of type $\vec n(x)$, $d_i(x)=a_{i+1}(x)-a_i(x)$ and
$$\ell:=\frac{k\chi-\sum_{x\in I}\sum^{l_x}_{i=1}d_i(x)r_i(x)}{r}.$$
Then $\sU_{C,\,\omega}$ is the GIT quotient $\sR^{ss}_{\omega}\xrightarrow{\psi} \sU_{C,\,\omega}:=\sU_C(r,d, \omega)$ and $\Theta_{\sR^{ss},\omega}$
descends to an ample line bundle $\Theta_{\sU_{C,\,\omega}}$ on $\sU_{C,\,\omega}$ when $\ell$ is an integer.

\begin{defn}\label{defn3.4} When $C$ is a smooth projective curve, let
$${\rm Det}: \sU_{C,\,\omega}\to J^d_C,\quad E\mapsto {\rm det}(E):=\bigwedge^rE$$
be the determinant map. Then, for any $L\in J^d_C$, the fiber
$${\rm Det}^{-1}(L):=\sU_{C,\,\omega}^L$$ is called moduli space of semistable parabolic bundles with a fixed determinant.
\end{defn}

Let $\sR^L_F\subset \sR$ be the sub-scheme of locally free sheaves with a fixed determinant $L$, and $(\sR^{ss}_{\omega})^L\subset\sR^{ss}_{\omega}$, $\,(\sR^{s}_{\omega})^L\subset\sR^s_{\omega}$
be the closed subsets of locally free sheaves with the fixed determinant $L$. Then $\sU_{C,\,\omega}^L$ is the GIT quotient
$(\sR^{ss}_{\omega})^L\xrightarrow{\psi}(\sR^{ss}_{\omega})^L//{\rm SL}(V):=\sU_{C,\,\omega}^L$. The proof of globally F-regular type of $\sU_{C,\,\omega}^L$ needs essentially the following two results.

\begin{prop}\label{prop3.5} Let $|{\rm I}|$ be the number
of parabolic points. Then, for any data $\omega=(k, \{\vec n(x),\,\,\vec a(x)\}_{x\in I})$, we have
\begin{itemize}
 \item[(1)]  $\,\,{\rm Codim}((\sR_{\omega}^{ss})^L\setminus (\sR_{\omega}^s)^L)\ge (r-1)(g-1)+\frac{1}{k}|{\rm
 I}|$,
\item[(2)]  $\,\,{\rm Codim} (\sR^L_F\setminus(\sR_{\omega}^{ss})^L)>(r-1)(g-1)+\frac{1}{k}|{\rm
I}|$.
\end{itemize}
\end{prop}

\begin{proof} This is in fact Proposition 5.1 of \cite{Su1} where we did not fix determinant and the term $\frac{1}{k}|{\rm
 I}|$ was omitted. However, the proof also works for the case of fixed determinant.
\end{proof}

\begin{prop}\label{prop3.6} Let $\omega_c=(2r, \{\vec n(x),\,\,\vec a_c(x)\}_{x\in I})$, where
$$\vec a_c(x)=(\bar a_1(x),\bar a_2(x),\cdots,\bar a_{l_x+1}(x))$$
satisfy $\bar a_{i+1}(x)-\bar a_i(x)=n_i(x)+n_{i+1}(x)$ ($1\le i\le l_x$). Then, when
\ga{3.1} {(r-1)(g-1)+\frac{|I|}{2r}\ge 2,}
the moduli space $\sU^L_{C,\,\omega_c}=(\sR^{ss}_{\omega_c})^L//{\rm SL}(V)$ is a normal Fano variety with only rational singularities.
\end{prop}

\begin{proof} It is in fact a reformulation of Proposition 2.2 of \cite{Su1} where a formula of anti-canonical bundle $\omega^{-1}_{\sR_F}$ (thus
a formula of $\omega^{-1}_{\sR^L_F}$) was given (see also Proposition 4.2 of \cite{Su3} for a tidier formula). The line bundle $\omega^{-1}_{\sR^L_F}$
is precisely determined by the data $\omega_c$ and descends to an ample line bundle $\Theta_{\sU^L_{C,\,\omega_c}}$, which is precisely $\omega^{-1}_{\sU^L_{C,\,\omega_c}}$ when $${\rm Codim}((\sR_{\omega_c}^{ss})^L\setminus (\sR_{\omega_c}^s)^L)\ge 2$$ by a result of F. Knop (see \cite{Kn}).
Thus we are done by the condition \eqref{3.1} and (1) of Proposition \ref{prop3.5}.
\end{proof}

\begin{thm}\label{thm3.7} The moduli spaces $\sU_{C,\,\omega}^L$ are of globally F-regular type. If Jacobian $J^0_C$ of $C$ is of F-split type,
so is $\sU_{C,\,\omega}$.
\end{thm}

\begin{proof} Choose a subset $I'\subset C$  such that $I'\cap I=\emptyset$ and
\ga{3.2} {(r-1)(g-1)+\frac{|I|+|I'|}{2r}\ge 2.}
Let $$\sR'=\underset{x\in I\cup I'}{\times_{\mathbf{Q}}}
Flag_{\vec n(x)}(\sF_x)=\sR\times_{\mathbf{Q}}\left(\underset{x\in I'}{\times_{\mathbf{Q}}}
Flag_{\vec n(x)}(\sF_x)\right)\xrightarrow{\hat f} \sR$$ be the projection, $\hat Y\subset\sR$ be the Zariski closure of $\sR^L_F$ and
$$\hat Z={\hat f}^{-1}(\hat Y)\subset \sR',\quad {\sR'}^L_F={\hat f}^{-1}(\sR^L_F)\subset\hat Z.$$
Then ${\sR'}^L_F\subset \hat Z$, $\sR^L_F\subset \hat Y$ are normal (in fact, smooth) ${\rm SL}(V)$-invariant open sub-schemes such that
$\hat Y^{ss}_{\omega}=(\sR^{ss}_{\omega})^L\subset \sR^L_F$,  $\hat Z^{ss}_{\omega'}=(\sR'^{ss}_{\omega'})^L\subset \sR'^L_F$
holds for any polarizations determined by data $\omega$, $\omega'$. It is clear that
$ \sR'^L_F\xrightarrow{\hat f}\sR^L_F$ is a flag bundle and $p$-compatible with $\hat f_*\sO_{\sR'^L_F}=\sO_{\sR^L_F}$. Thus
$\sR'^L_F\subset \hat Z$, $\sR^L_F\subset \hat Y$, $\sR'^L_F\xrightarrow{\hat f}\sR^L_F$
satisfy the conditions (1) and (2) of Proposition \ref{prop2.10}. To verify condition (3) in Proposition \ref{prop2.10}, let
$$W:=\hat Z^{s}_{\omega'}=(\sR'^{s}_{\omega'})^L\subset \sR'^L_F,\quad \hat X=\hat f^{-1}((\sR^{ss}_{\omega})^L)\cap W,$$
$\hat Z^{ss}_{\omega'}\xrightarrow{\varphi} Z:=\hat Z^{ss}_{\omega'}//{\rm SL}(V)$ and $X=\varphi(\hat X)\subset Z$. It is clear that
$$\hat X=\varphi^{-1}(X).$$
If we choose $\omega'=(2r, \{\vec n(x),\,\,\vec a_c(x)\}_{x\in I\cup I'})$ in Proposition \ref{prop3.6}, then
$${\rm Codim}({\sR'}^L_F\setminus W)\ge (r-1)(g-1)+\frac{|I|+|I'|}{2r}\ge 2$$
by Proposition \ref{prop3.5}, and $Z$ is a normal Fano variety with only rational singularities.
Thus $Z$ is of globally F-regular type by Proposition \ref{prop2.6}, so is $\sU_{C,\,\omega}^L=(\sR^{ss}_{\omega})^L//{\rm SL}(V)=\hat Y^{ss}_{\omega}//{\rm SL}(V)$
by Proposition \ref{prop2.10}.

If $J^0_C$ is of F-split type, so is $J^0_C\times \sU^L_{C,\,\omega}$. We have a $r^{2g}$-fold covering
$$J^0_C\times \sU^L_{C,\,\omega}\xrightarrow{f} \sU_{C,\,\omega}, \quad f(\sL_0, E)=\sL_0\otimes E$$
which implies that $\sU_{C,\,\omega}$ is of F-split type.
\end{proof}

\section{Globally $F$-regular type of Moduli spaces of generalized parabolic sheaves}

In this section, we prove that moduli spaces of generalized parabolic sheaves with a fixed
determinant on a smooth curve are of globally F-regular type. We will continue to use notations of last section.

Let $\{x_1,\,x_2\}\subset C\setminus I$ be two different points, a \emph{generalized parabolic sheaf} (GPS) $(E,Q)$ of rank $r$ and degree $d$ on $C$ consists of
a sheaf $E$ of degree $d$ on $C$, torsion free of rank $r$ outside
$\{x_1,x_2\}$ with parabolic structures at the points of $I$ and an $r$-dimensional quotient
$$E_{x_1}\oplus E_{x_2}\xrightarrow{q} Q\to 0.$$

\begin{defn}\label{defn3.8} A GPS $(E,Q)$ on an irreducible smooth curve $C$ is called \emph{semistable} (resp.,
\emph{stable}), if for every nontrivial subsheaf $E'\subset E$ such that
$E/E'$ is torsion free outside $\{x_1,x_2\},$ we have
$$par\chi(E')-dim(Q^{E'})\leq
r(E')\cdot\frac{par\chi(E)-dim(Q)}{r(E)} \,\quad (\text{resp.,
$<$}),$$ where $Q^{E'}=q(E'_{x_1}\oplus E'_{x_2})\subset Q.$
\end{defn}

\begin{thm}[Theorem X2 of \cite{NR} or Theorem 2.24 of \cite{Su3} for arbitrary rank]\label{thm3.9} For any data $\omega=(k, \{\vec n(x),\,\,\vec a(x)\}_{x\in I})$, there exists
a normal projective variety $\sP_{\omega}$ with at most rational singularities, which is the
coarse  moduli space of $s$-equivalence classes of semi-stable GPS on $C$ with parabolic structures
at the points of $I$ given by the data $\omega$.
\end{thm}

Recall the construction of $\sP_{\omega}$. Let $Grass_r(\sF_{x_1}\oplus\sF_{x_2})\to \bold Q$ and
$$\widetilde{\sR}=Grass_r(\sF_{x_1}\oplus\sF_{x_2})\times_{\bold
Q}\sR\xrightarrow{\rho} \sR.$$  $\omega=(k, \{\vec n(x),\,\,\vec a(x)\}_{x\in I})$ determines
a polarization, which linearizes the ${\rm SL}(V)$-action on $\wt\sR$, such that the open set $\wt{\sR}^{ss}_{\omega}$ (resp. $\wt{\sR}^s_{\omega}$) of
GIT semistable (resp. GIT stable) points are precisely the set of semistable (resp. stable) GPS on $C$ (see \cite{Su3}). Then $\sP_{\omega}$ is the GIT quotient
\ga{3.3} {\wt{\sR}^{ss}_{\omega}\xrightarrow{\psi}\wt{\sR}^{ss}_{\omega}//{\rm SL}(V):=\sP_{\omega}.}

\begin{nota}\label{nota3.10} Let $\sH\subset\wt{\sR}$ be the open subscheme
parametrising the generalised parabolic sheaves $E=(E,E_{x_1} \oplus
E_{x_2}\xrightarrow{q}Q)$ satisfying \begin{itemize}
\item [(1)] the torsion ${\rm Tor}\,E$ of $E$ is
supported on $\{x_1,x_2\}$ and $$q:({\rm Tor}\,E)_{x_1}\oplus ({\rm
Tor}\,E)_{x_2}\hookrightarrow Q$$
\item [(2)] if $N$ is large enough, then
$H^1(E(N)(-x-x_1-x_2))=0$ for all $E$ and $x\in C$.
\end{itemize}
\end{nota}
Then $\sH$ is reduced, normal, Gorenstein with at most rational singularities (see Proposition 3.2 and Remark 3.1 of \cite{Su1}). Moreover,
for any data $\omega$, we have $\wt{\sR}^{ss}_{\omega}\subset \sH$ and, by Lemma 5.7 of \cite{Su1}, there is a morphism ${\rm Det}_{\sH}:\sH\to J^d_C$
which extends determinant morphism on open set $\wt\sR_F\subset\sH$ of locally free sheaves, and induces a flat morphism
\ga{3.4} {{\rm Det}: \sP_{\omega}\to J^d_C.}

\begin{nota}\label{nota3.11} For $L\in J^d_C$, let $\sH^L={\rm Det}^{-1}(L)\subset\sH$,
$$\wt\sR_F^L={\rm Det}^{-1}(L)\subset\wt\sR_F, \quad (\wt\sR_{\omega}^{ss})^L={\rm Det}^{-1}(L)\subset \wt\sR_{\omega}^{ss}.$$
Then $\sP^L_{\omega}={\rm Det}^{-1}(L)\subset \sP_{\omega}$ is the GIT quotient
$$(\wt\sR_{\omega}^{ss})^L\xrightarrow{\psi}\sP^L_{\omega}=(\wt\sR_{\omega}^{ss})^L//{\rm SL}(V).$$
\end{nota}

\begin{prop}[Proposition 5.2 of \cite{Su1}]\label{prop3.12}
Let $\sD_1^f=\hat\sD_1\cup\hat\sD_1^t$
and $\sD_2^f=\hat\sD_2\cup\hat\sD_2^t$, where $\hat\sD_i\subset \wt{\sR}$ is the Zariski closure of
$\hat\sD_{F,\,i}\subset\wt{\sR}_F$ consisting of $(E,Q)\in\wt{\sR}_F$ that $E_{x_i}\to Q$ is not an isomorphism, and
${\hat\sD}_1^t\subset\wt{\sR}$ (rep. ${\hat\sD}_2^t\subset\wt{\sR}$) consists of $(E,Q)\in\wt{\sR}$ such that $E$ is not locally
free at $x_2$ (resp. at $x_1$). Then \begin{itemize}
\item [(1)] ${\rm Codim}(\sH^L\setminus(\wt\sR_{\omega}^{ss})^L)>(r-1)g+\frac{|I|}{k};$
\item [(2)] the complement in $(\wt\sR_{\omega}^{ss})^L\setminus\{\sD_1^f\cup\sD^f_2\}$ of the set
$\wt\sR_{\omega}^{s}$ of stable points has codimension $\ge
(r-1)g+\frac{|I|}{k}$.
\item [(3)] ${\rm Codim}((\wt\sR^{ss}_{\omega})^L\setminus W_{\omega})\ge (r-1)g+\frac{|I|}{k}$, $W_{\omega}\subset (\wt\sR^{ss}_{\omega})^L$ defined by
$$W_{\omega}:=\left\{ (E, Q)\in (\wt\sR_{\omega}^{ss})^L\bigg|
\begin{aligned}
& \forall \,\,E'\subset E \ \ \text{with}\ \ 0<r(E')< r, \ \text{we have} \\& \frac{par\chi(E')-dim(Q^{E'})}{r(E')}<\frac{par\chi(E)-dim(Q)}{r(E)}
\end{aligned}\right\}.$$
\end{itemize}
\end{prop}

\begin{proof} The statements (1) and (2) are contained in Proposition 5.2 of \cite{Su1} (where the term $\frac{|I|}{k}$ was omitted). The proof of Proposition 5.2 (2) in \cite{Su1} implies statement (3) here.

\end{proof}

\begin{prop}\label{prop3.13} Let $\omega_c=(2r, \{\vec n(x),\,\,\vec a_c(x)\}_{x\in I})$ be the data in Proposition \ref{prop3.6}
and $\Theta_{J^d_C}$ be the theta line bundle on $J^d_C$. Assume
\ga{3.5} {(r-1)(g-1)+\frac{|I|}{2r}\ge 2.}
Then there is an ample line bundle $\Theta_{\sP_{\omega_c}}$ on $\sP_{\omega_c}$ such that
$$\omega^{-1}_{\sP_{\omega_c}}=\Theta_{\sP_{\omega_c}}\otimes {\rm Det}^*(\Theta_{J^d_C}^{-1}).$$
In particular, for any $L\in J^d_C$, $\sP^L_{\omega_c}$ is a normal Fano variety with only rational singularities.
\end{prop}

\begin{proof} Let
$V\otimes\sO_{C\times\sH}(-N)\to \sE\to 0$, $\,\,\sE_{x_1}\oplus\sE_{x_2}\to \sQ\to 0$ and
$$\{\,\,\sE_{\{x\}\times \sH}=\sQ_{\{x\}\times \sH,\,l_x+1}\twoheadrightarrow\sQ_{\{x\}\times \sH,\,l_x}
\twoheadrightarrow \cdots\twoheadrightarrow \sQ_{\{x\}\times\,\sH,1}
\twoheadrightarrow0\,\,\}_{x\in I}$$ be the universal quotients and universal flags. Let $\omega_{C}=\sO(\sum_qq)$ and
$$\Theta_{J^d_C}=(detR\pi_{J^d_{C}}\sL)^{-2}\otimes\sL_{x_1}^r\otimes\sL_{x_2}^r\otimes\sL_y^{2\chi-2r}\otimes\bigotimes_q\sL_q^{r-1}$$
where $\sL$ is the universal line bundle on $C\times
J^d_{C}$. Then we have
$$\aligned&\omega^{-1}_{\sH}=(det\,R\pi_{\sH}\sE)^{-2r}\otimes\\&
\bigotimes_{x\in I}\left\{(det\,\sE_x)^{n_{l_x+1}-r}
\otimes\bigotimes^{l_x}_{i=1}(det\,\sQ_{x,i})
^{n_i(x)+n_{i+1}(x)}\right\}\otimes(det\,\sQ)^{2r}\\&
\otimes(det\,\sE_y)^{2\chi-2r}\otimes{\rm Det}_{\sH}^*(\Theta_{J^d_{C}}^{-1}):=\hat\Theta_{\omega_c}\otimes {\rm Det}_{\sH}^*(\Theta_{J^d_{C}}^{-1})
\endaligned$$ by Proposition 3.4 of \cite{Su1}, and $\hat\Theta_{\omega_c}$ descends to an ample
line bundle $\Theta_{\sP_{\omega_c}}$ on $\sP_{\omega_c}$ (see Lemma 2.3 of \cite{Su1}). Thus
$$(\psi_*\omega_{\wt{\sR}^{ss}_{\omega_c}}^{-1})^{inv.}=\Theta_{\sP_{\omega_c}}\otimes{\rm Det}^*(\Theta_{J^d_C}^{-1}).$$
When condition \eqref{3.5} holds, the lower bounds in Proposition \ref{prop3.5} and Proposition \ref{prop3.12} are at least two. Thus
Lemma 5.6 of \cite{Su1} is applicable (where assumption $g\ge 2$ in Lemma 5.6 of \cite{Su1} is replaced by condition \eqref{3.5})
and $(\psi_*\omega_{\wt{\sR}^{ss}_{\omega_c}}^{-1})^{inv.}=\omega^{-1}_{\sP_{\omega_c}}$.
\end{proof}

\begin{thm}\label{thm3.14} For any data $\omega=(k, \{\vec n(x),\,\,\vec a(x)\}_{x\in I})$, the moduli space
$\sP^L_{\omega}$ is of globally F-regular type.
\end{thm}

\begin{proof} Choose a finite subset $I'\subset C\setminus I$ satisfying \eqref{3.5}. Recall that
$$\sR=\underset{x\in I}{\times_{\bold Q}}Flag_{\vec n(x)}(\sF_x),\quad \sR'=\underset{x\in I\cup I'}{\times_{\mathbf{Q}}}
Flag_{\vec n(x)}(\sF_x)\xrightarrow{\hat f} \sR$$ be the projection and
$\widetilde{\sR}=Grass_r(\sF_{x_1}\oplus\sF_{x_2})\times_{\bold Q}\sR\xrightarrow{\rho} \sR$. Let
\ga{3.6} {\wt\sR':=Grass_r(\sF_{x_1}\oplus\sF_{x_2})\times_{\bold Q}\sR'\xrightarrow{\hat f} \wt\sR.}
Then, on $\sH^L\subset \wt\sR$,  it is clear that $(\sH')^L:={\hat f}^{-1}(\sH^L)\xrightarrow{\hat f} \sH^L$ is a ${\rm SL}(V)$-invariant and
$p$-compatible morphism such that $\hat f_*\sO_{(\sH')^L}=\sO_{\sH^L}$.

For any data $\omega=(k, \{\vec n(x),\,\,\vec a(x)\}_{x\in I})$, $\omega_c=(2r, \{\vec n(x),\,\,\vec a_c(x)\}_{x\in I\cup I'})$, we have
$(\wt\sR_{\omega}^{ss})^L\subset\sH^L$, $\,(\wt\sR_{\omega_c}^{\prime\,ss})^L\subset (\sH')^L$. Recall
$$(\wt\sR_{\omega}^{ss})^L\xrightarrow{\psi}\sP^L_{\omega}:=Y,\quad (\wt\sR_{\omega_c}^{\prime\,ss})^L\xrightarrow{\varphi}\sP^L_{\omega_c}:=Z.$$
To apply Proposition \ref{prop2.10}, let $W=W_{\omega_c}\subset (\wt\sR_{\omega_c}^{\prime\,ss})^L$ and
$$\hat X=W\cap\hat f^{-1}((\wt\sR_{\omega}^{ss})^L).$$
By Proposition \ref{prop3.12},  ${\rm Codim}((\sH')^L\setminus W)\ge (r-1)g+\frac{|I|+|I'|}{2r}\ge 2$. Thus it is enough to
check the condition that $\hat X=\varphi^{-1}\varphi(\hat X)$. This is equivalent (see Remark 1.2 of \cite{Su1}) to prove that
\ga{3.7}{\forall\,\,(E, Q)\in (\wt\sR_{\omega_c}^{\prime\,ss})^L,\quad (E,Q)\in \hat X\,\,\,\Leftrightarrow\,\,\, gr(E,Q)\in \hat X.}
In fact, for any $(E, Q)\in (\wt\sR_{\omega_c}^{\prime\,ss})^L$, it is clear that we have
$$(E, Q)\in W \,\,\,\Leftrightarrow\,\,\,gr(E,Q)=(\wt E, \wt Q)\oplus (\,_{x_1}\tau_1\oplus\,_{x_2}\tau_2, \tau_1\oplus\tau_2)\in W$$
where $(\wt E,\wt Q)$ is a stable GPS (see Definition 1.5 of \cite{Su1}). Thus either
$$0\to (\,_{x_1}\tau_1\oplus\,_{x_2}\tau_2, \tau_1\oplus\tau_2)\to (E,Q)\to (\wt E,\wt Q)\to 0$$
or $0\to (E',Q')\to (E, Q)\to (\,_{x_i}\mathbb{C},\mathbb{C})\to 0$. Then $(E,Q)$ is semi-stable (respect to $\omega$) if and only if
$(\wt E,\wt Q)$ is semi-stable (respect to $\omega$). Thus \eqref{3.7} is proved and we are done.
\end{proof}

When $C=C_1\cup C_2$ is reducible with two smooth irreducible
components $C_1$ and $C_2$ of genus $g_1$ and $g_2$ meeting at only
one point $x_0$ (which is the only node of $C$), we fix an ample
line bundle $\sO(1)$ of degree $c$ on $C$ such that
$deg(\sO(1)|_{C_i})=c_i>0$ ($i=1,2$). For any coherent sheaf $E$,
$P(E,n):=\chi(E(n))$ denotes its Hilbert polynomial, which has degree
$1$. We define the rank of $E$ to be
$$r(E):=\frac{1}{deg(\sO(1))}\cdot \lim \limits_{n\to\infty}\frac{P(E,n)}
{n}.$$ Let $r_i$ denote the rank of the restriction of $E$ to $C_i$
($i=1,2$), then
$$P(E,n)=(c_1r_1+c_2r_2)n+\chi(E),\quad r(E)=
\frac{c_1}{c_1+c_2}r_1+\frac{c_2}{c_1+c_2}r_2.$$ We say that $E$ is
of rank $r$ on $X$ if $r_1=r_2=r$, otherwise it will be said of rank
$(r_1,r_2)$.

Fix a finite set $I=I_1\cup I_2$ of smooth points on $C$, where
$I_i=\{x\in I\,|\,x\in C_i\}$ ($i=1,2$), and parabolic data $\omega=\{k,\vec
n(x),\vec a(x)\}_{x\in I}$ with
$$\ell:=\frac{k\chi-\sum_{x\in I}\sum^{l_x}_{i=1}d_i(x)r_i(x)}{r}$$
(recall $d_i(x)=a_{i+1}(x)-a_i(x)$, $r_i(x)=n_1(x)+\cdots+n_i(x)$). Let
\ga{3.8}{n^{\omega}_j=\frac{1}{k}\left(r\frac{c_j}{c_1+c_2}\ell+\sum_{x\in
I_j}\sum^{l_x}_{i=1}d_i(x)r_i(x)\right)\,\,\,(j=1,\,\,2).}

\begin{defn}\label{defn3.15} For any coherent sheaf $F$ of rank $(r_1,r_2)$, let
$$m(F):= \frac{r(F)-r_1}{k}\sum_{x\in I_1}a_{l_x+1}(x)+
\frac{r(F)-r_2}{k}\sum_{x\in I_2}a_{l_x+1}(x),$$ the modified
parabolic Euler characteristic and slop of $F$ are
$${\rm par}\chi_m(F):={\rm par}\chi(F)+m(F),\quad {\rm par}\mu_m(F):=\frac{{\rm par}\chi_m(F)}{r(F)}.$$
A parabolic sheaf $E$ is called semistable (resp. stable) if, for
any subsheaf $F\subset E$ such $E/F$ is torsion free, one has, with
the induced parabolic structure,
$${\rm par}\chi_m(F)\le \frac{{\rm par}\chi_m(E)}{r(E)}r(F)\quad (resp.<).$$
\end{defn}

\begin{thm}[Theorem 1.1 of \cite{Su2} or Theorem 2.14 of \cite{Su3}]\label{thm3.16} There
exists a reduced, seminormal projective scheme
$$\sU_C:=\sU_C(r,d,\sO(1),
\{k,\vec n(x),\vec a(x)\}_{x\in I_1\cup I_2})$$ which is the coarse
moduli space of $s$-equivalence classes of semistable parabolic
sheaves $E$ of rank $r$ and $\chi(E)=\chi=d+r(1-g)$ with parabolic structures
of type $\{\vec n(x)\}_{x\in I}$ and weights $\{\vec a(x)\}_{x\in
I}$ at points $\{x\}_{x\in I}$. The moduli space $\sU_C$ has at most
$r+1$ irreducible components.
\end{thm}

The normalization of $\sU_C$ is a moduli space of semistable GPS on $\wt C=C_1\bigsqcup C_2$ with parabolic structures at points $x\in I$. Recall

\begin{defn}\label{defn3.17} A GPS $(E,E_{x_1}\oplus E_{x_2}\xrightarrow{q}Q)$ is called semistable (resp.,
stable), if for every nontrivial subsheaf $E'\subset E$ such that
$E/E'$ is torsion free outside $\{x_1,x_2\},$ we have, with the
induced parabolic structures at points $\{x\}_{x\in I}$,
$$par\chi_m(E')-dim(Q^{E'})\leq
r(E')\cdot\frac{par\chi_m(E)-dim(Q)}{r(E)} \,\quad (\text{resp.,
$<$}),$$ where $Q^{E'}=q(E'_{x_1}\oplus E'_{x_2})\subset Q.$
\end{defn}

\begin{thm}[Theorem 2.1 of \cite{Su2} or Theorem 2.26 of \cite{Su3}]\label{thm3.18} For any data $\omega=(\{k,\vec n(x),\,\,\vec a(x)\}_{x\in I_1\cup I_2},\sO(1))$,
the coarse  moduli space $\sP_{\omega}$ of $s$-equivalence classes of semi-stable GPS on $\wt C$ with parabolic structures
at the points of $I$ given by the data $\omega$ is a disjoint union of at most $r+1$ irreducible, normal projective varieties $\sP_{\chi_1,\chi_2}$
( $\chi_1+\chi_2=\chi+r$, $n_j^{\omega}\le\chi_j\le n_j^{\omega}+r$) with at most rational singularities.
\end{thm}

For fixed $\chi_1$, $\chi_2$ satisfying $\chi_1+\chi_2=\chi+r$ and $n_j^{\omega}\le\chi_j\le n_j^{\omega}+r$ ($j=1,\,2$), recall the construction
of $\sP_{\omega}=\sP_{\chi_1,\chi_2}$. Let
$$P_i(m)=c_irm+\chi_i,\quad \sW_i=\sO_{C_i}(-N),\quad V_i=\Bbb
C^{P_i(N)}$$ where
$\sO_{C_i}(1)=\sO (1)|_{C_i}$ has degree $c_i$. Consider the Quot schemes $\textbf{Q}_i=Quot(V_i\otimes\sW_i,
P_i)$, the universal quotient
$V_i\otimes\sW_i\to \sF^i\to 0$ on $C_i\times \textbf{Q}_i$ and
the relative flag scheme
$$\sR_i=\underset{x\in I_i}{\times_{\textbf{Q}_i}}
Flag_{\vec n(x)}(\sF^i_x)\to \textbf{Q}_i.$$ Let
$\sF=\sF^1\oplus\sF^2$ denote direct sum of pullbacks of $\sF^1$,
$\sF^2$ on $$\wt C\times
(\textbf{Q}_1\times\textbf{Q}_2)=(C_1\times\textbf{Q}_1)\sqcup(C_2\times\textbf{Q}_2).$$
Let $\sE$ be the pullback of $\sF$ to $\wt
C\times(\sR_1\times\sR_2)$, and
$$\rho:\widetilde{\sR}=Grass_r(\sE_{x_1}\oplus\sE_{x_2})\to\sR=\sR_1\times\sR_2\to
\textbf{Q}=\textbf{Q}_1\times\textbf{Q}_2.$$
For the given $\omega=(\{k,\vec n(x),\,\vec a(x)\}_{x\in I_1\cup I_2},\mathcal{O}(1))$, let
$\wt{\sR}_{\omega}^{ss}$ (resp. $\wt{\sR}_{\omega}^{s}$) denote the open set of GIT semi-stable (resp. GIT stable) points under
action of $G=({\rm GL}(V_1)\times {\rm GL}(V_2))\cap {SL}(V_1\oplus V_2)$ on $\wt\sR$ respect to the polarization determined by $\omega$.
Let $\sH\subset\wt{\sR}$ be the open set defined in Notation \ref{nota3.10}, then for any data $\omega$ we have
$$\wt{\sR}_{\omega}^{s}\subset\wt{\sR}_{\omega}^{ss}\subset\sH.$$
The moduli space in Theorem \ref{thm3.18} is nothing but the GIT quotient
$$\psi:\wt{\sR}_{\omega}^{ss}\to \sP_{\omega}:=\wt{\sR}_{\omega}^{ss}//G.$$

There exists a morphism $\hat{\rm Det}_{\sH}: \sH\to J^d_{\wt C}=J^{d_1}_{C_1}\times J^{d_2}_{C_2}$, which extends
$$\hat{\rm Det}_{\sH_F}: \sH_F\to J^{d_1}_{C_1}\times J^{d_2}_{C_2},\quad (E,Q)\mapsto ({\rm det}(E|_{C_1}), {\rm det}(E|_{C_2}))$$
on the open set $\sH_F\subset \sH$ of GPB (i.e. GPS $(E,Q)$ with $E$ locally free) and induces a flat determinant morphism
$${\rm Det}_{\sP_{\omega}}:\sP_{\omega}\to J^d_{\wt C}=J^{d_1}_{C_1}\times J^{d_2}_{C_2}$$
(see page 46 of \cite{Su3} for detail). In fact, for any $L\in J^d_{\wt C}=J^{d_1}_{C_1}\times J^{d_2}_{C_2}$, let
\ga{3.9} {\sP_{\omega}^L:={\rm Det}_{\sP_{\omega}}^{-1}(L)\subset \sP_{\omega}}
and note that abelian variety $J^0_C=J^0_{C_1}\times J^0_{C_2}$ acts on $\sP_{\omega}$, the induced morphism
$\sP^L_{\omega}\times J^0_X\to \sP_{\omega}$ is a finite cover (see the proof of Lemma 6.6 in \cite{Su3}). Similarly, let $\sH^L=\hat{\rm Det}_{\sH}^{-1}(L)$ and $(\wt{\sR}_{\omega}^{ss})^L=\wt{\sR}_{\omega}^{ss}\cap \sH^L$, then
$$\psi:(\wt{\sR}_{\omega}^{ss})^L\to \sP^L_{\omega}=(\wt{\sR}_{\omega}^{ss})^L//G.$$

We do not have good estimate of ${\rm Codim}(\sH\setminus \wt{\sR}_{\omega}^{ss})$ since sub-sheaves $(E_1, Tor(E_2))$, $(Tor(E_1), E_2)$
of $E=(E_1,E_2)$ with rank $(r,0)$, $(0,r)$ may destroy semi-stability of $(E,Q)$ where $Tor(E_i)\subset E_i$ ($i=1,\,2$) are torsion sub-sheaves.
But we have estimate of ${\rm Codim}(\sH_{\omega}\setminus \wt{\sR}_{\omega}^{ss})$, where
$$\sH_{\omega}=\left\{\aligned&\text{$(E,Q)\in\sH$, with
$n_j^{\omega}\le \chi(E_j)=\chi_j\le n_j^{\omega}+r$ ($j=1,\,2$), and}\\&\text{${\rm dim}({\rm Tor}(E_1))\le n^{\omega}_2+r-\chi_2$, $\,\,{\rm dim}({\rm Tor}(E_2))\le n^{\omega}_1+r-\chi_1$}\endaligned\right\}.$$

\begin{prop}\label{prop3.19}
Let $\sD_1^f=\hat\sD_1\cup\hat\sD_1^t$
and $\sD_2^f=\hat\sD_2\cup\hat\sD_2^t$, where $\hat\sD_i\subset \wt{\sR}$ is the Zariski closure of
$\hat\sD_{F,\,i}\subset\wt{\sR}_F$ consisting of $(E,Q)\in\wt{\sR}_F$ that $E_{x_i}\to Q$ is not an isomorphism, and
${\hat\sD}_1^t\subset\wt{\sR}$ (rep. ${\hat\sD}_2^t\subset\wt{\sR}$) consists of $(E,Q)\in\wt{\sR}$ such that $E$ is not locally
free at $x_2$ (resp. at $x_1$). Then \begin{itemize}
\item [(1)] ${\rm Codim}(\sH_{\omega}^L\setminus(\wt\sR_{\omega}^{ss})^L)>
\underset{1\le i\le 2}{{\rm min}}\left\{(r-1)(g_i-\frac{r+3}{4})+\frac{|I_i|}{k}\right\};$
\item [(2)] ${\rm Codim}((\wt\sR_{\omega}^{ss})^L\setminus\{\sD_1^f\cup\sD^f_2\}\setminus(\wt\sR_{\omega}^{s})^L)>\underset{1\le i\le 2}{{\rm min}}\left\{(r-1)(g_i-1)+\frac{|I_i|}{k}\right\}$ when $n_1^{\omega}<\chi_1<n_1^{\omega}+r$;
\item [(3)] ${\rm Codim}((\wt\sR_{\omega}^{ss})^L\setminus\{\sD_1^f\cup\sD^f_2\}\setminus W_{\omega})\ge\underset{1\le i\le 2}{{\rm min}}\left\{(r-1)(g_i-1)+\frac{|I_i|}{k}\right\}$
when $\chi_1=n_1^{\omega}$ or $n_1^{\omega}+r$, where
$$W_{\omega}:=\left\{(E,Q)\in(\wt\sR_{\omega}^{ss})^L \bigg|\begin{aligned}
& \text{$\frac{par\chi(E')-dim(Q^{E'})}{r(E')}<\frac{par\chi(E)-dim(Q)}{r(E)}$}\\
&\text{$\forall$ $E'\subset E$ of rank $(r_1,r_2)\neq (0,r)$, $(r,0)$, $(0,0)$}\end{aligned}\right\};$$
\item [(4)] ${\rm Codim}((\wt\sR^{ss}_{\omega})^L\setminus W_{\omega})\ge \underset{1\le i\le 2}{{\rm min}}\left\{(r-1)(g_i-\frac{r+3}{4})+\frac{|I_i|}{k}\right\}$.
\end{itemize}
\end{prop}

\begin{proof} The statements (1), (2) and (3) are in fact reformulations of Proposition 6.3 in \cite{Su3} where determinants are not fixed. (4) follows the proof of
Proposition 6.3 in \cite{Su3} (see Remark 6.7 (2) of \cite{Su3}).
\end{proof}

\begin{prop}\label{prop3.20} For any $\omega$, let $\wt{\sR}_{\omega}^{ss}\xrightarrow{\psi} \sP_{\omega}:=\wt{\sR}_{\omega}^{ss}//G$ and assume
$$\underset{1\le i\le 2}{{\rm min}}\left\{(r-1)(g_i-\frac{r+3}{4})+\frac{|I_i|}{k}\right\}\ge 2.$$
Then $(\psi_*\omega_{\wt{\sR}_{\omega}^{ss}})^{inv.}=\omega_{\sP_{\omega}}$. For
$\omega_c=(2r,\{\vec n(x),\,\vec a_c(x)\}_{x\in I_1\cup I_2})$ satisfying
$$\underset{1\le i\le 2}{{\rm min}}\left\{(r-1)(g_i-\frac{r+3}{4})+\frac{|I_i|}{2r}\right\}\ge 2,$$
there is an ample line bundle $\Theta_{\sP_{\omega_c}}$ on $\sP_{\omega_c}$ such that
$$\omega^{-1}_{\sP_{\omega_c}}=\Theta_{\sP_{\omega_c}}\otimes {\rm Det}_{\sP_{\omega_c}}^*(\Theta_{J^d_{\wt C}}^{-1}).$$
In particular, for any $L\in J^d_{\wt C}$, $\sP^L_{\omega_c}$ is a normal Fano variety with only rational singularities.
\end{prop}

\begin{proof} According to a result of Knop in \cite{Kn} (see Lemma 4.17 of \cite{NR} for its global formulation), to prove $(\psi_*\omega_{\wt{\sR}_{\omega}^{ss}})^{inv.}=\omega_{\sP_{\omega}}$, it is enough to show that (1) the subset where the action of $G$ is not free has codimension at least two;
(2) for every prime divisor $D$ in $\wt{\sR}_{\omega}^{ss}$, $\psi(D)$ has codimension at most $1$.

To verify condition (1), when $n_1^{\omega}<\chi_1<n_1^{\omega}+r$, we have
$${\rm Codim}(\wt\sR_{\omega}^{ss}\setminus\{\sD_1^f\cup\sD^f_2\}\setminus\wt\sR_{\omega}^{s})>\underset{1\le i\le 2}{{\rm min}}\left\{(r-1)(g_i-1)+\frac{|I_i|}{k}\right\}\ge 2.$$
Note that $\sD^f_j=\hat\sD_j\cup \hat\sD^t_j$ where $\hat\sD_j$, $\hat\sD^t_j$ are irreducible, normal subvarieties (see Proposition C.7 of \cite{NR}), and the subsets of $\hat\sD_j$ and $\hat\sD^t_j$, where the action of $G$ is free, are open subsets.
Thus it is enough to find a $(E, Q)\in\sD^f_j$ ($j=1,\,2$) such that its automorphisms are only scales.
Let $E_i'$ ($i=1,\,2$) be stable parabolic bundles of rank $r$ and $\chi(E'_i)=\chi_i'$ on $C_i$ with parabolic structures determined by $(k, \{\vec n(x),\vec a(x)\}_{x\in I_i})$. If we take $\chi_1'=\chi_1-1$, $\chi_2'=\chi_2$, let $E_1=E_1'\oplus\,_{x_1}\mathbb{C}\cdot\beta$, $E_2=E_2'$ and
$E=(E_1,E_2)$, the surjection $E_{x_1}\oplus E_{x_2}\xrightarrow{q} Q$ is defined by any isomorphism $E_{x_2}\xrightarrow{q_2} Q$ and
a linear map $E_{x_1}=(E'_1)_{x_1}\oplus \mathbb{C}\cdot\beta\xrightarrow{q_1} Q$ such that $q_1(\beta)\neq 0$ and $q_1|_{(E'_1)_{x_1}}\neq 0$. Then $(E,Q)\in {\hat\sD}_2^t$ by definition.
To see ${\rm Aut}((E,Q))=\mathbb{C}^*$, let $0\to K\to E_{x_1}\oplus E_{x_2}\xrightarrow{q}Q\to 0$ and $(E,Q)\xrightarrow{\Phi}(E,Q)$ be an isomorphism. Then $E\xrightarrow{\Phi}E$ is an isomorphism of parabolic bundles such that $\Phi_{x_1+x_2}(K)=K$. Since $E_1'$, $E_2$ are stable, $\Phi|_{E_1}=(\lambda'_1,\lambda_1):E_1'\oplus\,_{x_1}\mathbb{C}\beta\to E_1'\oplus\,_{x_1}\mathbb{C}\beta$ and $\Phi|_{E_2}=\lambda_2: E_2\to E_2$ for nonzero constants
$\lambda_1',\,\lambda_1,\,\lambda_2$. The requirement $\Phi_{x_1+x_2}(K)=K$ implies that $\lambda'_1=\lambda_1=\lambda_2$.
In fact, $K=\{\,(\alpha, f(\alpha))\in E_{x_1}\oplus E_{x_2}\,|\,\forall\,\alpha\in E_{x_1}\}$ where $f=-q_2^{-1}q_1: E_{x_1}\to E_{x_2}$. For any $\alpha=\alpha'+\beta\in E_{x_1}$, we have
$$\Phi_{x_1+x_2}(\alpha, f(\alpha))=(\lambda'_1\alpha'+\lambda_1\beta, \lambda_2f(\alpha')+\lambda_2f(\beta))\in K$$
which implies that $\lambda_2f(\alpha')+\lambda_2f(\beta)=\lambda'_1f(\alpha')+\lambda_1f(\beta)$. Thus $\lambda_1=\lambda_2$ (by taking $\alpha'=0$)
and $\lambda_1'=\lambda_2$ (by taking $\alpha'$ such that $q_1(\alpha')\neq 0$).
Similarly, one can find such $(E,Q)\in  {\hat\sD}_1^t$. To construct $(E, Q)\in \hat\sD_j$ with ${\rm Aut}(E,Q)=\mathbb{C}^*$, we take $\chi'_i=\chi_i$, $E_i=E_i'$ and
$E=(E_1, E_2)$ with $$E_{x_1}\oplus E_{x_2}\xrightarrow{q}Q\to 0$$
defined by any isomorphism $q_2: E_{x_2}\to Q$ and nontrivial linear map $q_1:E_{x_1}\to Q$ (which is not surjective). Thus $(E,Q)\in \hat\sD_1$ and ${\rm Aut}(E,Q)=\mathbb{C}^*$. Similarly one can find such $(E,Q)\in \hat\sD_2$. When $\chi_1=n_1^{\omega}$ or $n_1^{\omega}+r$, we have
${\rm Codim}(\wt\sR_{\omega}^{ss}\setminus\{\sD_1^f\cup\sD^f_2\}\setminus W_{\omega})\ge 2$. Thus we only need to show, for any
$(E,Q)\in (\wt\sR_{\omega}^{ss}\setminus\{\sD_1^f\cup\sD^f_2\})\cap W_{\omega}$, ${\rm Aut}((E,Q))=\mathbb{C}^*$. This is easy since the proof of
Lemma 6.1 (4) in \cite{Su3} implies stability of parabolic bundles $E_1$ and $E_2$. Thus any automorphism of $E=(E_1,E_2)$ must be of type
$(\lambda_1id_{E_1},\lambda_2id_{E_2})$, which induces an automorphism of $(E,Q)$ if and only if $\lambda_1=\lambda_2$.

To verify condition (2), if a prime divisor $D$ is not contained in $\wt\sR_{\omega}^{ss}\setminus W_{\omega}$, $\psi(D)$ is a divisor. If $D$ is contained
in $\wt\sR_{\omega}^{ss}\setminus W_{\omega}$, then $D$ must be one of $\hat\sD_j$, $\hat\sD^t_j$ since ${\rm Codim}(\wt\sR_{\omega}^{ss}\setminus\{\sD_1^f\cup\sD^f_2\}\setminus W_{\omega})\ge 2$. However, $\psi(\hat\sD_j)=\psi(\hat\sD^t_j)=\sD_j$ ($j=1,\,2$)
by Proposition 2.5 of \cite{Su2}, which are divisors. Thus we have proved that $(\psi_*\omega_{\wt{\sR}_{\omega}^{ss}})^{inv.}=\omega_{\sP_{\omega}}$.

When $\omega=\omega_c$, by Proposition 6.4 of \cite{Su3}, there is an ample line bundle $\Theta_{\sP_{\omega_c}}$ on $\sP_{\omega_c}$ such that
$\omega_{\wt{\sR}_{\omega_c}^{ss}}^{-1}=\psi^*(\Theta_{\sP_{\omega_c}}\otimes {\rm Det}_{\sP_{\omega_c}}^*(\Theta_{J^d_{\wt C}}^{-1})).$ Thus
$$\omega_{\sP_{\omega_c}}=(\psi_*(\omega_{\wt{\sR}_{\omega_c}^{ss}}))^{inv.}=\Theta_{\sP_{\omega_c}}^{-1}\otimes {\rm Det}_{\sP_{\omega_c}}^*(\Theta_{J^d_{\wt C}})$$
and $\omega^{-1}_{\sP^L_{\omega_c}}=\Theta_{\sP_{\omega_c}}|_{\sP^L_{\omega_c}}$ is ample, $\sP^L_{\omega_c}$ is a normal Fano variety with only rational singularities.
\end{proof}

\begin{lem}\label{lem3.21} Let $V$ be a normal variety acting by a reductive group $G$. Suppose a good quotient $\phi:V\to U$ exists. Let $\sL$ be a line bundle on $U$ and
$\wt \sL=\phi^*(\sL)$. Let $V''\subset V'\subset V$ be open $G$-invariant subvarieties of $V$ such that $\phi(V')=U$ and $V''=\phi^{-1}(U'')$ for some nonempty open subset
$U''\subset U$. Then $\phi^G_*(\wt\sL|_{V'})=\sL$ (i.e, for any nonempty open set $X\subset U$,
${\rm H}^0(\phi^{-1}(X),\wt\sL)^{inv.}\to {\rm H}^0(V'\cap\phi^{-1}(X),\wt\sL)^{inv.}$ is an isomorphism).

\end{lem}

\begin{proof} It is in fact a reformulation of Lemma 4.16 in \cite{NR}, where
$${\rm H}^0(V,\wt\sL)^{inv.}\to {\rm H}^0(V',\wt\sL)^{inv.}$$ was shown to be an isomorphism.
\end{proof}

\begin{thm}\label{thm3.22} For any data $\omega=(\{k,\vec n(x),\,\,\vec a(x)\}_{x\in I_1\cup I_2},\sO(1))$ and integers $\chi_1$, $\chi_2$ satisfying $\chi_1+\chi_2=\chi+r$, $n_j^{\omega}\le\chi_j\le n_j^{\omega}+r$ ($j=1,\,2$), let $\sP^L_{\omega}$ be
the coarse  moduli space  of $s$-equivalence classes of semi-stable GPS $E=(E_1, E_2)$ on $\wt C$ with fixed determinant $L$, $\chi(E_j)=\chi_j$ and parabolic structures at the points of $I$ given by the data $\omega$. Then $\sP^L_{\omega}$ is of globally $F$-regular type.
\end{thm}

\begin{proof} Let $I'_i\subset X_i\setminus (I_i\cup\{x_i\})$ be a subset and $I'=I'_1\cup I'_2$. Recall
$$\sR_i=\underset{x\in I_i}{\times_{\textbf{Q}_i}}
Flag_{\vec n(x)}(\sF^i_x)\to \textbf{Q}_i$$ and
$\rho:\widetilde{\sR}=Grass_r(\sF^1_{x_1}\oplus\sF^2_{x_2})\to\sR=\sR_1\times\sR_2$, let
$$\sR_i'=\underset{x\in I_i\cup I_i'}{\times_{\mathbf{Q}}}
Flag_{\vec n(x)}(\sF^i_x)\to \sR_i,\quad \sR'=\sR_1'\times\sR_2'\xrightarrow{\hat f}\sR=\sR_1\times \sR_2$$ be the projection and
$\wt\sR':=\wt\sR\times_{\sR}\sR'\xrightarrow{\hat f}\wt\sR$ be induced via the diagram
$$\CD
  {\wt\sR}' @>\rho>> {\sR}' \\
  @V \hat f VV @V \hat f VV  \\
  \wt\sR @>\rho>> \sR
\endCD$$
Then, on $\sH^L\subset \wt\sR$,  it is clear that $(\sH')^L:={\hat f}^{-1}(\sH^L)\xrightarrow{\hat f} \sH^L$ is a $G$-invariant and
$p$-compatible morphism such that $\hat f_*\sO_{(\sH')^L}=\sO_{\sH^L}$.

For $\omega=(k, \{\vec n(x),\,\vec a(x)\}_{x\in I},\sO(1))$, $\omega_c=(2r, \{\vec n(x),\,\vec a_c(x)\}_{x\in I\cup I'},\sO(1))$, we have
$(\wt\sR_{\omega}^{ss})^L\subset\sH^L$, $\,(\wt\sR_{\omega_c}^{\prime\,ss})^L\subset (\sH')^L$. Moreover, for $\omega_c$,
let $\ell_j^c=2\chi_j-r-\sum_{x\in I_j\cup I'_j}r_{l_x}(x)$ and $\ell^c=\ell^c_1+\ell^c_2=2\chi-\sum_{x\in I\cup I'}r_{l_x}(x).$
Then $$\sum_{x\in I\cup I'}\sum^{l_x}_{i=1} (\bar a_{i+1}(x)-\bar a_i(x))r_i(x) +r\ell^c= 2r\chi.$$
The choices of $\{\vec n(x)\}_{x\in I'}$ satisfying $\ell^c_j=\frac{c_j}{c_1+c_2}\ell^c$ for arbitrary large $|I'_1|$ and $|I'_2|$ are possible and it is easy to compute that $n_j^{\omega^c}=\chi_j-\frac{r}{2}$, thus
$$n_j^{\omega^c}<\chi_j<n_j^{\omega^c}+r\qquad (j=1,\,\,2).$$
Recall $(\wt\sR_{\omega}^{ss})^L\xrightarrow{\psi}\sP^L_{\omega}:=Y,\quad (\wt\sR_{\omega_c}^{\prime\,ss})^L\xrightarrow{\varphi}\sP^L_{\omega_c}:=Z$, choose $I_i'$ satisfying
$$\underset{1\le i\le 2}{{\rm min}}\left\{(r-1)(g_i-\frac{r+3}{4})+\frac{|I_i|+|I_i'|}{2r}\right\}\ge 2.$$
Then $Z$ is a normal Fano variety with only rational singularities by Proposition \ref{prop3.20}, which is in particular of globally $F$-regular type.
To apply Proposition \ref{prop2.10}, let $$W=W_{\omega_c}\subset (\wt\sR_{\omega_c}^{\prime\,ss})^L,\quad \hat X=W\cap\hat f^{-1}((\wt\sR_{\omega}^{ss})^L).$$
For any $(E, Q)\in (\wt\sR_{\omega_c}^{\prime\,ss})^L\setminus(\wt\sR_{\omega_c}^{\prime\,s})^L$, there is an exact sequence
\ga{3.10}{0\to (E',Q')\to (E,Q)\to (\wt E,\wt Q)\to 0}
in the category $\sC_{\mu}$ (see Proposition 2.4 of \cite{Su2}) such that $(\wt E, \wt Q)$ is stable (respect to $\omega_c$). Then either $\wt E$ is torsion free when $r(\wt E)>0$ or $(\wt E,\wt Q)=(\,_{x_i}\mathbb{C}, \mathbb{C})$. If $(E,Q)\in W$, $\wt E$ has rank $r$ or rank $(r, 0)$, $(0,r)$ when $r(\wt E)>0$. Thus
it is easy to show that $(E,Q)\in W$ if and only if $gr(E,Q)$ is one of the following\begin{itemize}
\item [(1)] $gr(E,Q)=(\wt E, \wt Q)\oplus (\,_{x_1}\tau_1\oplus\,_{x_2}\tau_2, \tau_1\oplus\tau_2)$ where $(\wt E,\wt Q)\in\sC_{\mu}$ is stable of rank $(r,r)$;
\item [(2)] $gr(E,Q)=(\wt E_1,\wt Q_1)\oplus (\wt E_2,\wt Q_2)\oplus (\,_{x_1}\tau_1\oplus\,_{x_2}\tau_2, \tau_1\oplus\tau_2)$
where $(\wt E_1,\wt Q_1)$ and $(\wt E_2,\wt Q_2)\in \sC_{\mu}$ are stable of rank $(r,0)$ and $(0,r)$,\end{itemize}
which implies that $\varphi^{-1}\varphi(W)=W$. Hence, to check that $\hat X=\varphi^{-1}\varphi(\hat X)$, it is enough to show that $(E,Q)$
is semi-stable (respect to $\omega$) if and only if the above GPS
$(\wt E,\wt Q)$, $(\wt E_1,\wt Q_1)$ and $(\wt E_2,\wt Q_2)$ in (1) and (2) are semi-stable (respect to $\omega$) with the same slope $\mu_{\omega}(E,Q)$, which is easy to
check by using \eqref{3.10} when either $\wt E$ has rank $r$ or $r(\wt E)=0$. If $\wt E$ has rank $(0,r)$, $E'$ must have rank $(r,0)$. Then $(E,Q)$ is $\omega$-semistable if and only if $(E',Q')$, $(\wt E,\wt Q)$ are $\omega$-semistable with $\mu_{\omega}(E',Q')=\mu_{\omega}(\wt E,\wt Q)=\mu_{\omega}(E,Q)$ since the exact sequence
\eqref{3.10} is split in this case.

Now $\hat X \xrightarrow{\varphi} X:=\varphi(\hat X)\subset Z$ is a category quotient and the $G$-invariant $(\sH')^L\xrightarrow{\hat f} \sH^L$ induces a morphism $f: X \to Y$ such that
$$\CD
 (\sH')^L\supset\hat X @>\varphi>> X \\
  @V \hat f VV @V f VV  \\
  \sH^L\supset(\wt\sR_{\omega}^{ss})^L @>\psi>> Y
\endCD$$ is a commutative diagram. However we do not have $${\rm Codim}((\sH')^L\setminus W)\ge 2$$ as required in Proposition \ref{prop2.10}, which was used to prove
$f_*\sO_X=\sO_Y$.

On the other hand, let $(\wt\sR_{\omega, F}^{ss})^L\subset (\wt\sR_{\omega}^{ss})^L$, $(\sH_F')^L \subset (\sH')^L$ be the $G$-invariant open sets of GPS $(E,Q)$ with $E$ being locally free. Then
$${\rm Codim}((\sH_F')^L\setminus W)\ge\underset{1\le i\le 2}{{\rm min}}\left\{(r-1)(g_i-\frac{r+3}{4})+\frac{|I_i|+|I_i'|}{k}\right\}\ge 2$$
by Proposition \ref{prop3.19} (1) and (2), which and Lemma \ref{lem3.21} imply
$$f_*\sO_X=\sO_Y.$$
In fact, let $\hat X_F=\hat X\cap (\sH_F')^L$ and apply Lemma \ref{lem3.21} to the surjections
$$\hat X_F\xrightarrow\varphi X, \quad (\wt\sR_{\omega, F}^{ss})^L\xrightarrow\psi Y,$$
we have $\varphi^G_*\sO_{\hat X_F}=\sO_X$, $\psi^G_*\sO_{(\wt\sR_{\omega, F}^{ss})^L}=\sO_Y$, which means $\forall \,\,U\subset Y$,
$$\sO_Y(U)=H^0(\psi^{-1}(U),\sO_{(\wt\sR_{\omega,F}^{ss})^L})^{inv.}$$ $$\sO_X(f^{-1}(U))=H^0(\varphi^{-1}(f^{-1}(U)),\sO_{\hat X_F})^{inv.}.$$
On the other hand, for the $G$-invariant morphism $(\sH')^L\xrightarrow{\hat f}\sH^L$ with $\hat f_*\sO_{(\sH')^L}=\sO_{\sH^L}$,
by using the fact that $$\hat f^{-1}\psi^{-1}(U)\setminus\hat f^{-1}\psi^{-1}(U)\cap\hat X=\hat f^{-1}\psi^{-1}(U)\cap ((\sH_F')^L\setminus W)$$
has at least codimension two, we have
$$\aligned\sO_Y(U)&=H^0(\psi^{-1}(U),\sO_{(\wt\sR_{\omega,F}^{ss})^L})^{inv.}=H^0(\hat f^{-1}\psi^{-1}(U),\sO_{(\sH_F')^L})^{inv.}\\&=
H^0(\hat f^{-1}\psi^{-1}(U)\cap\hat X,\sO_{(\sH_F')^L})^{inv.}\\&=H^0(\varphi^{-1}(f^{-1}(U)),\sO_{\hat X_F})^{inv.}=\sO_X(f^{-1}(U)).
\endaligned$$
Thus $Y=\sP^L_{\omega}$ is of globally $F$-regular type since $X$ is so.
\end{proof}

\bibliographystyle{plain}

\begin{thebibliography}{99}


\bibitem{Br} M. Brion and S. Kumar: {\em Frobenius Splitting Methods in Geometry and Representation Theory},
Progress in Mathematics, {\bf 231}, Birkh{\"a}user Boston Inc. MA, 2005.

\bibitem{Gr} A. Grothendieck: {\em EGA IV}, Publications de IHES, {\bf
28} (1966), 5--255.

\bibitem{Kn} F. Knop: {\em Der kanonische Moduleines Invariantenrings}, Journal of Algebra, {\bf
127} (1989), 40--54.

\bibitem{KLT} S. Kumar, N. Lauritzen and J. F. Thomsen: {\em Frobenius splitting of cotangent bundles of flag varieties},
Invent. Math., {\bf 136} (1999), 603--621.

\bibitem{LRT} N. Lauritzen, U. Raben-Pedersen and J. F. Thomsen: {\em Global F-regularity of Schubert varieties with applications to
D-modules}, J. Amer. Math. Soc., {\bf 19} (2006), 345--355.

\bibitem{MeR} V. B. Mehta and T. R. Ramadas: {\em Moduli of vector bundles, Frobenius splitting, and invariant theory},
Ann. of Math. {\bf 144} (1996), 269--313.

\bibitem{MeRa} V. B. Mehta and A. Ramanathan: {Frobenius splitting and cohomology vanishing for Schubert varieties},
Ann. of Math. {\bf 122} (1985), 27--40.

\bibitem{MeRa88} V. B. Mehta and A. Ramanathan: {Schubert varieties in $G/B\times G/B$},
Compositio Math. 67 (1988), 355--358.

\bibitem{Mu} D. Mumford, J. Fogarty and F. Kirwan: {\em Geometric Invariant Theory},
Ergebnisse der Mathematik und ihrer Grenzgebiete, {\bf 34}, Springer-Verlag, 1994.

\bibitem{NR} M.S.Narasimhan and T.R. Ramadas: {\em Factorisation of generalised theta functions I},
Invent. Math., {\bf 114} (1993), 217--235.

\bibitem{Pa} C. Pauly: {\em Espaces de modules de fibr{\'e}s paraboliques et blocs conformes},
Duke Math. Journal, {\bf 84} (1996), 565--623.

\bibitem{RR}S. Ramanan and A. Ramanathan: {\em Projective normality of flag varieties and Schubert varieties},
Invent. Math., {\bf 79} (1985), 217--224.

\bibitem{Sc} K.Schwede and K.E.Smith: {\em Globally F-regular and log Fano varieties}, Advanced in Mathematics.
{\bf224}, (2010), 863--894.

\bibitem{Se} C.S.Seshadri: {\em Geometric reductivity over arbitrary base}, Adv. Math. {\bf26} (1977), 225--274.

\bibitem{Sm} K.E.Smith: {\em Globally F-regular varieties: Applications to vanishing theorems for quotients of Fano varieties},
Michigan Math J., {\bf48}, (2000), 553--572.

\bibitem{Su1} X. Sun: {\em Degeneration of moduli spaces and generalized theta functions},
Journal of Algebraic Geometry, {\bf 9}, (2000), 459--527.

\bibitem{Su2} X. Sun: {\em Factorization of generalized theta functions in the reducible case},
Ark. Mat., {\bf 41} (2003), 165--202.

\bibitem{Su3} X. Sun: {\em Factorization of generalized theta functions revisited},
Algebra Colloquium., {\bf 24} (2017), no.1, 1--52.



\end{thebibliography}

\renewcommand\refname{References}

\end{document}